\documentclass{amsart}
\usepackage{amssymb}
\usepackage{color}
\usepackage{float}
\usepackage[all,cmtip]{xy}
\usepackage{graphicx}
\usepackage{hyperref}
\usepackage{cancel}
\usepackage{pstricks}
\usepackage{mathabx}
\usepackage{tikz-cd}

\DeclareRobustCommand{\SkipTocEntry}[5]{}

\newcommand{\Acal}{{\mathcal{A}}}

\newcommand{\Ccal}{{\mathcal{C}}}
\newcommand{\Dcal}{{\mathcal{D}}}
\newcommand{\Ecal}{{\mathcal{E}}}

\newcommand{\Mcal}{{\mathcal{M}}}
\newcommand{\Ncal}{{\mathcal{N}}}

\newcommand{\Pcal}{{\mathcal{P}}}

\newcommand{\CC}{\mathbb{C}}

\newcommand{\NN}{\mathbb{N}}

\newcommand{\RR}{\mathbb{R}}

\newcommand{\ZZ}{\mathbb{Z}}

\newcommand{\iso}{\cong}

\newcommand{\id}{\textbf{\textit{I}}}

\newcommand{\Vol}{\operatorname{vol}}
\newcommand{\vol}{\operatorname{vol}}

\newcommand{\RFH}{\operatorname{RFH}}
\newcommand{\RFC}{\operatorname{RFC}}

\newcommand{\Crit}{\operatorname{Crit}}

\renewcommand{\id}{\operatorname{id}}

\newcommand{\haat}{\widehat}
\newcommand{\tiilde}{\widetilde}

\newtheorem{theorem}{Theorem}
\newtheorem*{theorem*}{Theorem}
\newtheorem{proposition}{Proposition}[section]
\newtheorem{lemma}[proposition]{Lemma}

\newtheorem{corollary}[theorem]{Corollary}

\theoremstyle{definition}
\newtheorem{definition}[proposition]{Definition}
\newtheorem{example}[proposition]{Example}

\newtheorem{assumption}[proposition]{Assumption}
\newtheorem{criterion}[proposition]{Criterion}

\theoremstyle{remark}
\newtheorem{remark}[proposition]{Remark}
\newtheorem{question}[proposition]{Question}

\numberwithin{equation}{section}

\begin{document}

\title[]{$C^0$-stability of topological entropy for Contactomorphisms}

\author{Lucas Dahinden}
\address{Universität Heidelberg}
\email{l.dahinden@gmail.com}


\date{\today}

\begin{abstract}
	Topological entropy is not lower semi-continous: small perturbation of the dynamical system can lead to a collapse of entropy. In this note we show that for some special classes of dynamical systems (geodesic flows, Reeb flows, positive contactomorphisms) topological entropy at least is stable in the sense that there exists a nontrivial continuous lower bound, given that a certain homological invariant grows exponentially.
\end{abstract}

\maketitle

\tableofcontents

\section{Introduction and results}



\subsubsection*{Topological entropy}
    Topological entropy $h_{\rm top}(\varphi)$ is a good numerical measure for the complexity of a self-map $\varphi:M\to M$ of a compact metrizeable space $M$. We use the following definition of Bowen~\cite{B71} and Dinaburg~\cite{D71}: Fix a metric $d$ generating the topology of $M$. A $(K,\delta)$-separated set is a subset $N\subseteq M$ such that for all $n\neq n'\in N$ there is a $k\in[0,K]$ such that $d(\varphi^k(n),\varphi^k(n'))\geq \delta$. Topological entropy is defined as the growth rate of maximal cardinality of $(T,\delta)$-separated sets:
$$h_{\rm top}(\varphi)= \sup_{\delta>0} \Gamma( \mbox{maximal cardinality of a $(T,\delta)$-separated set} ),$$
where for a sequence of non-negative numbers $a_k$, $\Gamma(a_k):=\limsup\frac1k\log a_k$ is the exponential growth rate.


\subsubsection*{General setup}
		Let $(M,\Lambda,\alpha)$ be a closed contact manifold with closed Legendrian $\Lambda$ which is fillable by the triple $(W,L,\lambda)$ consisting of a Liouville domain with asymptotically conical exact Lagrangian $L$. This means that $M=\partial W, \Lambda=\partial L= L\cap M, \alpha=\lambda|_{TM}$. The Examples~\ref{example} for Theorem~\ref{thm2} satisfy these conditions.
		
    We parametrize the class of positive paths of contactomorphisms by the set of positive contact Hamiltonians 
    $$\Pcal=\{h\in C^\infty(M\times[0,1],\RR)\mid 0< h\}.$$
    The functions $h$ define the contact Hamiltonian vector fields by 
    $$\alpha(X_h)=h;\quad \iota_{X_h}d\alpha=-dh+dh(R_\alpha)\alpha,$$
    and these vector fields generate the paths $\varphi^t_h$ by 
    $$\varphi^0_h=id;\quad \dot\varphi^t_h=X_h(\varphi^t_h).$$
		A positive contactomorphism is the endpoint of a positive path of contactomorphisms. We can extend $\varphi^t$ to a twisted periodic flow by convex combination with a Reeb flow. This means that we change the generating vector field $X$ without changing $\varphi^1$ such that it coincides with a Reeb vector field $X=R$ near $t=0$ and $t=1$, and thus a twisted periodic extension is smooth.
    Note that the function $h=1$ generates the Reeb flow of $\alpha$. The stability then comes from the growth of the positive part of the filtered Lagrangian Rabinowitz--Floer homology $\iota_\infty\RFH_+^T(W,L,h)$ that survives to infinity and which is defined for the subset $\Pcal_{\rm reg}$ of Hamiltonians for which $\bigcup_{t\neq 0}\varphi_h^t(\Lambda)\pitchfork \Lambda$ (For more on $\RFH_+^T(W,L,h)$ see Section~\ref{sec:poscont} and Appendix~\ref{appendix}, for the definition of $\iota_\infty$ see $(2)$ on page~\pageref{infinitePmod}).	

		The following theorem is the main result in~\cite{D18}. 
	
    \begin{theorem*}
    	If for some $h\in\Pcal_{\rm reg}$ the positive Lagrangian Rabinowitz--Floer homology $\RFH_+^T(W,L,h)$ has positive dimensional growth, then for $h\in\Pcal_{\rm reg}$ we have
    	$$h_{\rm top}(\varphi_h^1)>0.$$
    \end{theorem*}
		
		The main theorem of this article is the addition of stability in the sense of Definition~\ref{stab} below. 
		
		\begin{theorem}[Main theorem]\label{thm2}
    	Under the same assumptions, denote the dimensional growth of  $\RFH_+^T(W,L,h)$ by $\Gamma_{\RFH}(h):=\Gamma(\dim\iota_\infty\RFH^{<T}(h))$. Then
    	$$h_{\rm top}(\varphi_h^1)\geq \Gamma_{\rm RFH}(h).$$
    	This lower bound extends continuously to all of $\Pcal$, is positive and stable in the $C^0$-topology on $\Pcal$: 
    	$$\Gamma_{\RFH}(h)\geq \min_{x,t} h(x,t)\cdot\Gamma_{\rm RFH}(1)>0.$$
    \end{theorem}
		
		See Example~\ref{example} for a list of examples for which the theorem applies.
				
	\begin{remark}
		Our proof holds for the much more general class of (not necessarily twisted periodic) positive paths of contactomorphisms, as long as we impose uniform bounds $0<c\leq h\leq C$. However, the conclusions only hold for volume growth and we do not know the connection to topological entropy. For more on this, see Remark~\ref{nonautonomous}.
	\end{remark}



With the method of proof of Theorem~\ref{thm2} we are able to conclude a statement on \emph{orderability} in the sense of~\cite{EP00}: A contact manifold is orderable if there is no positive loop of contactomorphisms. 
\begin{corollary}\label{cor}
	Under the same assumptions as Theorem~\ref{thm2}, $(M,\alpha)$ is orderable.
\end{corollary}
\begin{proof}
	Choose an auxiliary Riemannian metric on $M$, all volumes will be measured with respect to this metric. 
	
	If $(M,\alpha)$ is not orderable, then there is a positive path of contactomorphisms $\varphi^t$ that starts and ends at the identity. We can extend $\varphi^t$ to a periodic family of positive contactomorphisms. The restriction to $\Lambda$ is the positively transverse family of Legendrians $\varphi^t(\Lambda)$, which has bounded volume since it is periodicity. 
	
	On the other hand, if the assumptions of Theorem~\ref{thm2} are satisfied, the proof in Section~\ref{sec:proofProp} reveals that the volume of $\varphi^t(\Lambda)$ grows exponentially.
\end{proof}
\begin{remark}
	In~\cite{AF12} the statement is already concluded for cosphere bundles. The proof there relies on strict inequalities of spectral invariants. The proof given here by volume growth might be of interest because of its simplicity.	
\end{remark}
	In~\cite{AM18} it is concluded that already non-vanishing of \emph{closed string} Rabinowitz--Floer homology of Liouville fillable contact manifolds implies orderability. This leads to the following question:
\begin{question}
	Does non-vanishing of Lagrangian Rabinowitz--Floer homology imply orderability? 
\end{question}
	Resolving this question might involve establishing an intrinsic relationship between Lagrangian and closed string Rabinowitz--Floer homology, in analogy to the better understood relationship between wrapped homology and symplectic homology through the closed-open map~\cite{Abo10}.

\subsection{Stability of topological entropy}

\subsubsection*{Computability of topological entropy} Unfortunately, topological entropy is very hard to compute explicitly, see~\cite{HKC92,M02,GHRS19}. One reason for this is the lack of lower semi-continuity of the function $\varphi\mapsto h_{\rm top}(\varphi)$ that associates to a map its entropy. The examples in Section~\ref{sec:examples} show that even for smooth diffeomorphisms of 3-dimensional closed manifolds, $h_{\rm top}$ can collapse to 0 under perturbation, i.e.\ can jump up from 0 when passing to a limit. Only in very special situations the topological entropy is continuous. For example, in the class of diffeomorphisms of closed 2-dimensional surfaces, topological entropy is continuous due to the existence and stability of hyperbolic sets, see e.g.~\cite[S5.6]{HK}. For further stability results see~\cite{HSX} or~\cite{DFPV} and the references therein.

    \begin{remark}[Upper semi-continuity]
        On the class of smooth maps, $h_{\rm top}$ is upper semi-continuous in the $C^\infty$ topology, see~\cite{N}. This fails on the class of $C^r$ maps. 
    \end{remark}

\subsubsection*{Stability of topological entropy}	
    Since there is no hope for lower semi-continuity of topological entropy, we aim at the next best thing: we look for an interesting class of spaces and maps for which the topological entropy is ``stable'', i.e.\ there is a nontrivial continuous lower bound for $h_{\rm top}$. Given stability, it is still hard to explicitly compute entropy, but one can estimate it from below, since total collapse of entropy under perturbation is prevented. 
	
	To formalize this, let $\Pcal$ be an appropriately topologized parameter space which continuously parametrizes a subset of the space of smooth self-maps $\Pcal\to\Dcal\subseteq C^\infty(M,M)$ equipped with the $C^0$-topology.
	
\begin{definition}\label{stab}
Topological entropy is \emph{stable} on $\Pcal$ if there exists a not identically vanishing continuous function 
	$$\gamma:\Pcal\to\RR_{\geq 0} \mbox{ such that }\gamma(p)\leq h_{\rm top}(\varphi_p).$$
\end{definition}
	Note that points in $\Dcal$ where $\gamma(p)=h_{\rm top}(\varphi_p)$ are points of lower semi-continuity of topological entropy in the $C^0$-topology and therefore also points of lower semi-continuity of topological entropy in $C^\infty$-topology. Since topological entropy is upper semi-continuous in $C^\infty(M,M)$ with $C^\infty$-topology, points where $\gamma(p)=h_{\rm top}(\varphi_p)$ are points of continuity of topological entropy on $\Dcal$ in $C^\infty$-topology.
	
	\begin{remark}
		The main focus of this article lies on (a subset of) the space of smooth contact Hamiltonians parametrizing (a subset of) the space of smooth contactomorphisms with $C^0$-topology. In order for this parametrization to be continuous, we must topologize the space of contact Hamiltonians with at least $C^1$-topology. However, the stability function provided by Theorem~\ref{thm2} is continuous even in $C^0$-topology. 
	\end{remark}
	
	\begin{remark}
		Given stability of topological entropy, one still can collapse topological entropy by large perturbations. See e.g.~\cite{AAS20}, where it is proven that every contact structure on a closed manifold admits a sequence of contact forms with fixed volume whose Reeb flows have topological entropy converging to 0. This is complementary to the present article, where we show stability under small perturbations. 
	\end{remark}

	
\subsubsection*{A criterion for stability}
	We work in the setting of maps $\varphi$ which are time-1 maps of twisted periodic flows, a class of flows containing autonomous flows: let $\varphi^t$ be a smooth 1-parameter family of maps such that $\varphi^0=id$ and $\dot\varphi^t=X^t$, where $X^t$ is a 1-periodic vector field. Then, iterations $(\varphi)^k$ of $\varphi:=\varphi^1$ coincide with $\varphi^t$ at integer times $t=k$. We formulate a criterion that implies stability of topological entropy of $\varphi$. It is used in all the examples given at the end of the introduction. 
	
	Let a class of twisted periodic flows $\varphi^t$ of a closed manifold $M$ be parametrized by $\Pcal$. Our main assumption is the abundance of chords from a fixed submanifold to a generic fiber of some fibered region in the manifold:
	
	\begin{criterion}\label{ass}
	    There is a submanifold $A^k\subseteq M$ of dimension $k$ and a subset $\Ncal\subseteq M$, where $\pi:\Ncal\to B$ is a fibration over a manifold $B^{k+1}$ of dimension $k+1$ such that the following holds. There is a continuous positive function $\alpha:\Pcal\to\RR_{>0}$ such that for almost all $b\in B$ the number $n(A,b,T,\varphi^t)$ of $\varphi^t$-chords from $A$ to the fiber over $b$ of length $\leq T$ is finite and grows uniformly $\alpha$-exponentially in $T$, i.e.\ for some $T_0=T_0(\varphi^t)$ we have for $T\geq T_0$:
		$$n(A,b,T,\varphi^t)\geq e^{\alpha T}.$$
      Note that then the exponential growth rate of $n(A,b,T,\varphi^t)$ is at least $\alpha$. We assume the number $T_0$ to be uniform in $b$, but not necessarily continuous in $\Pcal$.
	\end{criterion}
	
	\begin{remark}
		For the proof of Theorem~\ref{thm1} below there is no restriction on the dimension $k$ of the submanifold $A$. However, for the more specialized situation of Theorem~\ref{thm2} the submanifold $A$ will always be a Legendrian, i.e.\ if $M$ is $2n+1$-dimensional, then $k=n$. Examples where $k$ is less than half-dimensional can be created in a meaningful way by raising the codimension of $A$: Let Criterion~\ref{ass} be satisfied by $(M,A,\Ncal\stackrel{\pi}\to B,\Dcal)$ and let $M$ be a submanifold in a higher dimensional closed manifold $W$. Then, Criterion~\ref{ass} is also satisfied by $(W,A,\tiilde\Ncal\stackrel{\tiilde\pi}\to B,\tiilde\Dcal)$, where $\tiilde\Ncal=\Ncal\oplus\nu(M)\stackrel{\tiilde\pi}\to B$ is the Whitney sum of $\pi$ with the normal bundle of $M$ and $\tiilde\Dcal=\{\varphi\in C^{\infty}(W,W)\mid \varphi(M)=M, \varphi|_{M}\in\Dcal\}$ is the set of smooth extensions.
	\end{remark}
	
	The following result has been used in specific situations multiple times before. 
	We give a proof for the following general formulation in Section~\ref{sec:proofProp}.
	
	\begin{theorem}\label{thm1}
	  	Given a class of smooth self-maps $\varphi$ that are time-1 maps of twisted periodic flows $\varphi^t$ of a compact manifold
	  	$M$ which satisfy Criterion~\ref{ass}, then $\alpha$ constitutes a stability function for the topological entropy of $\varphi=\varphi^1$.
	\end{theorem}

\subsection{Examples}


In the following we give a list of examples where the Criterion~\ref{ass} is satisfied. All but the last also fit into the framework of Theorem~\ref{thm2}, although some results may be shown by more classical means. 
\begin{example}\label{example}
    As a first example let $Q$ be a closed manifold with exponentially growing fundamental group. The sets in Criterion~\ref{ass} are 
		\begin{itemize}
			\item $\Pcal=\{\mbox{Riemannian metrics on $Q$}\}$,
			\item $\Dcal=\{\varphi^1\mid \varphi^t\mbox{ is a geodesic flow on }SQ\}$, 
			\item $A=S_pQ$, $B\subseteq Q$ is a neighbourhood of $p$ and $\Ncal=\pi^{-1}(B)$. 
		\end{itemize}
	The number of geodesic chords from $A$ to $\pi^{-1}(q)$ of length at most $R$ is then at least the number of elements of $\pi_1(Q,p,q)$ that are represented by paths that lift into the ball $\tiilde B_R(p)\subseteq\tiilde Q$ in the universal cover, since the minimizer of length is a geodesic. The function $\gamma$ is then given by the growth of this number, which is independent of $q$ and positive if the group growth of $\pi_1(Q)$ is positive. Continuity of $\gamma$ with respect to $g$ in $C^0$-norm comes from the fact that the balls $\tiilde B_R(p)$ vary continuously with $g$. This is a classic discussion: We define the ball volume growth\footnote{In Riemanniann geometry this is called volume growth. We use the different term to avoid confusion with $\Gamma_{\rm vol}$.} as the growth of volume of a ball in the universal cover of $Q$: $\Gamma_{\rm ball}(M,g)=\Gamma(\Vol(\tiilde B_R(p))$. For the relation between $\Gamma_{\rm ball}$ and the group growth of $\pi_1$ and for the fact that $\Gamma_{\rm ball}$ is a lower bound for topological entropy, see~\cite{M79}. The example also fits into the framework of Theorem~\ref{thm2}: The Liouville domain is the sublevel set $\{H(p,q)<1\}\subseteq T^*Q$ in the cotangent bundle, where $H=\frac12 \|p\|^2$ is the Hamiltonian associated to the Riemannian metric. Incidentally, it is the Hamiltonian associated to the contact Hamiltonian $h\equiv 1$ on the boundary. The Lagrangian $L=T^*_qQ$ is the fiber over the point $q$. The filtered Rabinowitz--Floer homology groups split into a product over the fundamental group: 
	$$\iota_\infty\RFH^{<T}(h)=\prod_{\alpha\in\pi_1(Q)}\iota_\infty\RFH^{<T}_\alpha(h).$$ 
	The number of $\alpha\in\pi_1(Q)$ represented by a geodesic of length $<\frac12R^2$ is estimated from below by the number of nontrivial factors $\iota_\infty\RFH^{<\frac12R^2}_\alpha(h)$.
		    
		A bit more intricate is the situation if the fundamental group of $Q$ is finite and the homology of the based loop space grows exponentially. The function $\gamma$ is then given by the growth $\Gamma(\dim \iota_\infty HM^{\leq T}(\Ecal))$ of the Morse homology of the set of loops in $M$ of energy at most $T$. As Gromov~\cite{Gro78} showed, positivity of $\gamma$ is a topological invariant of $Q$. Again, this can be expressed in terms of Theorem~\ref{thm2}, which leads immediately to the following generalization that was first explored by Frauenfelder--Macarini--Schlenk~\cite{MacSch11,FraSch06}: In terms of Criterion~\ref{ass} we have the sets $\Pcal=\{\mbox{Contact forms on $(S^*Q,\xi_{\rm std})$}\}$ and $\Dcal=\{\varphi^1\mid \varphi^t\mbox{ is a Reeb flow on }S^*Q\}$. Using the Abbondandolo--Schwarz isomorphism from the Morse homology of the based loop space to Lagrangian Floer homology, one finds the function $\gamma$ is given by $\gamma(\alpha)=C\min \frac{\alpha_0}{\alpha}$, where the constant $C$ is the dimensional growth of the Lagrangian Floer homology of a reference contact form $\alpha_0$. This setup can be further transported into the framework of Theorem~\ref{thm2} by a Lagrangian analog~\cite{D18} of the Cieliebak--Frauenfelder--Oancea isomorphism.

    A further extension of the above results was given by Alves--Meiwes~\cite{AM17} to boundaries of Liouville domains with exponentially growing wrapped Floer homology (the open string analog to symplectic homology). The role of $A$ is then taken by some Legendrian sphere that is fillable by an exact Lagrangian, and $\Ncal$ is a neighbourhood of $A$ which is fibered by Legendrian spheres. In this framework Alves and Meiwes managed to construct many examples of contact manifolds different from cosphere bundles that admit a nonvanishing function $\gamma$. Among the examples are (non-standard) contact spheres of dimension $\geq 7$, $S^3\times S^2$, and non-standard contact structures on any plumbing of cosphere bundles of base dimension $\geq 4$.

    All the above examples were generalized by the first author to the class of positive contactomorphisms~\cite{D18}. As a subclass we find lightlike geodesic flows on globally hyperbolic Lorentz manifolds, which might be of independent interest, see~\cite{D20}.

    As a final example we mention Reeb flows on 3-manifolds with a Legendrian knot whose cylindrical contact homology has exponential growth, studied by Alves~\cite{Alv16a}. This situation arises in hyper-tight contact manifolds with pseudo-Anosov fundamental group. 
\end{example}

\addtocontents{toc}{\SkipTocEntry}
\subsection*{Acknowledgments}
	I wish to thank Leonid Polterovich and Felix Schlenk for inputs and advice. This work is supported by SFB/TRR 191 `Symplectic Structures in Geometry, Algebra and Dynamics' funded by the DFG.

\section{Examples for total collapse of topological entropy}\label{sec:examples}

The following example by Milnor~\cite{M02} shows that topological entropy is not lower semi-continuous, even in very simple settings: it fails for smooth maps from the closed unit disk to itself.
\begin{example}
	The family $f_t:D\subseteq\CC\mapsto D;\;z\mapsto tz^2$ of maps is smooth. Looking at the restriction $f_1|_{\partial D}:\partial D\to\partial D$ one sees that $h_{\rm top}(f_1)=\log 2$. However, $f_t$ has zero entropy for all $t<1$ since the chain recurrent set of $f_t$ is the origin.
\end{example}
	If one imposes that the surface be closed and the map be a $C^r, r\geq1,$ diffeomorphism, we have lower semi-continuity because of the existence and local stability of hyperbolic sets, as described in~\cite[S5.6]{HK}.
	However, this result is 2-dimensional in essence and ceases to hold in higher dimensions.
	The following example shows the lack of lower semi-continuity of topological entropy in the category of $C^r$ diffeomorphisms $(r\geq 1$ or $r=\infty)$ of closed manifolds in three dimensions:
\begin{example}
	Let $\Sigma^2$ be a closed surface and $f:\Sigma\to\Sigma$ be a $C^r$ diffeomorphism with $h_{\rm top}(f)>0$ that is isotopic to the identity through a smooth path $f_s$ of $C^r$ diffeomorphisms. We can assume that $f_s=f$ for $s$ near $0$ and $f_s=id$ for $s$ near $1$. Let $\tau:T^1\to\RR$ be a bump function on the circle $T^1=\RR/\ZZ$ supported in $(0,\frac12)$ with $\tau(\frac14)=1$. Further, let $g_t$ be the negative gradient flow of a Morse function on $T^1$ with critical points only in $0$ and $\frac12$. Then
	$F_t:\Sigma\times T^1\to\Sigma\times T^1;\;(x,\theta)\mapsto(f_{\tau(\theta)}(x),g_t(\theta))$
	is a smooth family of $C^r$ diffeomorphisms such that $F_0=(f_{\tau(\theta)},\id)$ has positive topological entropy in the fiber $\theta=\frac14$ and such that $F_t$ has zero topological entropy for every $t>0$ since the restriction to the chain recurrent set $\Sigma\times\{0,\frac12\}$ of $F_t$ is the identity.
\end{example}

\section{Proof of Theorem~\ref{thm1}}\label{sec:proofProp}
    Here, we give a proof of the Stability Theorem~\ref{thm1} under the assumption of Criterion~\ref{ass}. The proof is based on the proofs in the special situations, see e.g.~\cite{FLS15,D18}. 
				
		In this paper, we always work with smooth maps $\varphi$ on compact manifolds $M$, so by a combined theorem of Yomdin~\cite{Y} and Newhouse~\cite{N}, topological entropy coincides with volume growth:
    $$h_{\rm top}(\varphi^t)=\Gamma_{\rm vol}(\varphi^t)=\sup_{S\subset M}\Gamma(\vol \varphi^t(S)),$$ 
    where the supremum is taken over all compact submanifolds  $S$ of arbitrary codimension and $\vol$ is taken with respect to any Riemannian metric. Since volume growth is independent of the choice of Riemannian metric, we can choose a nice one. Let $g_B$ be a Riemannian metric on the base $B^{k+1}$ and $\Vol^{k+1}_B$ its induced volume. Then we choose $g$ on $M$ such that the induced $k+1$-dimensional volume of $g$ in $\Ncal$ is larger than its shadow on $B$, i.e.\ 
		\begin{align}
        \Vol^{k+1}_g&\geq \pi^*\Vol^{k+1}_B.\label{projectionofvolume}
    \end{align}
    We denote the set of $b\in B$ for which the growth condition holds by $B_{reg}$, and we denote the trace left behind by $\varphi^tA$ by $A^T:=\bigcup_{t\in[0,T]}\varphi^tA$.
    
    The proof is completed in 6 steps:
    \begin{itemize}
    	\item[Step 1] $\Gamma_{\rm vol}(\varphi^t)\geq \Gamma(\Vol^k(\varphi^tA))$,
    	\item[Step 2] $\Gamma(\Vol^k(\varphi^tA))\geq\Gamma(\Vol^{k+1}(A^T))$
    	\item[Step 3] $\Gamma(\Vol^{k+1}(A^T))\geq \Gamma(\Vol^{k+1}(A^T)\cap\Ncal)$
    	\item[Step 4] $\Gamma(\Vol^{k+1}(A^T)\cap\Ncal)\geq \Gamma(\int_{A^T}\pi^*d\Vol_B)$
    	\item[Step 5] $\Gamma(\int_{A^T}\pi^*d\Vol_B)\geq \Gamma(\int_{B_{reg}} n(A,b,T,\varphi^t) d\Vol_B)$
    	\item[Step 6] $\Gamma(\int_{B_{reg}} n(A,b,T,\varphi^t) d\Vol_B)\geq \alpha$.
    \end{itemize}
    
    Step 1 holds since the supremum is an upper bound.
    
    Step 2 holds since the vector field $X$ generating $\varphi^t$ is bounded above since it is periodic on a compact space:
    $$\Vol^{k+1}(A^T)=\int_0^T \int_{\varphi^tA} \| pr_{(\varphi^tA)^\perp}X\| \;d\Vol^k \;dt\leq \sup \|X\|\int_0^T\Vol^k(\varphi^t A)\;dt.$$
		Then, we note that for $f$ with $\limsup_{t\to\infty}f(t)\geq 1$ we have that $\Gamma(\int_0^Tf(t)\;dt)\leq \Gamma(f(T))$.

    Step 3 holds since volume is monotone under inclusion and since $f\leq g$ implies $\Gamma(f)\leq\Gamma(g)$. Step 4 is an immediate consequence of~(\ref{projectionofvolume}). Step 5 is just a reformulation. 
    
    Step 6 is concluded since we assumed Criterion~\ref{ass} and thus for some $T_0$ we have $n(A,b,T,\varphi^t)\geq e^{\alpha T}$ for all $b\in B$ and $T\geq T_0$. This completes the proof.		

\begin{remark}[Non-autonomous dynamical systems]\label{nonautonomous}
For volume growth $\Gamma_{\rm vol}$ it makes sense to work with the much more general class of non-autonomous dynamical systems, i.e.\ for smooth families of smooth maps
$\varphi^t$ with $\varphi^0=id$, which are not necessarily twisted periodic. The above proof goes through as long as we impose that $X^t:=\dot\varphi^t$ is bounded from above. Also, the proof of Theorem~\ref{thm2} in Section~\ref{sec:poscont} does not rely on twisted periodicity, but only on the Hamiltonian $h$ being uniformly bounded for all time $0<c\leq h\leq C$. 

One can also define topological entropy in this more general setting: A $(T,\delta)$-separated set is a subset $N\subseteq M$ such that for all $n\neq n'\in N$ there is a $t\in[0,T]$ such that $d(\varphi^t(n),\varphi^t(n'))\geq \delta$. Topological entropy is then defined as the exponential growth rate of maximal cardinality of $(T,\delta)$-separated sets:
$$h_{\rm top}(\varphi^t)= \sup_{\delta>0} \Gamma( \mbox{maximal cardinality of a $(T,\delta)$-separated set} ).$$
This generalized notion of topological entropy for non-autonomous systems is studied only in a few papers. For example, in~\cite{KS96} it is shown that in the discrete time setting this definition of generalized topological entropy (which is analogous to the definition of Bowen and Dinaburg) coincides with the definition that generalizes the definition of topological entropy by Adler, Konheim and McAndrew. 
\end{remark}
For this article it is of interest under what conditions the theorems of Yomdin~\cite{Y} and Newhouse~\cite{N} generalize to non-autonomous dynamical systems. 
\begin{question}
How are volume growth and topological entropy related for non-autonomous dynamical systems? 
\end{question}

\section{Proof of Theorem~\ref{thm2}}\label{sec:poscont}

    In this section we show how to obtain Criterion~\ref{ass} in the setup of~\cite{D18}, which proves Theorem~\ref{thm2}. For an outline of the construction of Lagrangian Rabinowitz--Floer homology see the appendix. The main point is the setup of a persistence module for each element of $\Pcal$ and proving that a perturbation of the parameter changes the persistence module Lipschitz--continuously with respect to (logarithmized) interleaving distance. We refer to~\cite{PRSZ} for basic notions on persistence modules.

\subsubsection*{Liouville domains and conical Lagrangians}

    Let $(M,\Lambda,\alpha)$ be a contact manifold with Legendrian $\Lambda$ which is fillable by the triple $(W,L,\lambda)$ consisting of a Liouville domain $(W,\lambda)$ with asymptotically conical exact Lagrangian $L$. This means that $M=\partial W, \Lambda=\partial L= L\cap M, \alpha=\lambda|_{TM}$. 
		
\begin{example}
	The stereotypical example for this setup is the cosphere bundle of a compact manifold $(S^*Q,S_q^*Q,\lambda)$, where $\lambda=pdq$ is the tautological 1-form, together with the Legendrian being a fiber over a point.
	We can realize the cosphere bundle explicitly as the level set of a fiberwise starshaped Hamiltonian function $H$. We can then take the sublevel set as Liouville domain (i.e.\ the codisk bundle $D^*Q$), where a codisk fiber fills $S^*_qQ$ asymptotically conically.		
\end{example}
    
\subsubsection*{Positive contactomorphisms}
		In the introduction we parametrize the class of twisted periodic positive paths of contactomorphisms by positive periodic contact Hamiltonians. However, in view of Remark~\ref{nonautonomous} it makes sense to directly parametrize the set of positive paths of contactomorphisms by the set of bounded positive contact Hamiltonians 
    $$\Pcal=\{h\in C^\infty(M\times\RR,\RR)\mid \exists c,C\in\RR^+:0<c\leq h\leq C\}.$$
    The contact Hamiltonian vector field of such a function generates the smooth family of contactomorphisms $\varphi^t_h$, which are uniformly positively transverse to the contact structure $\ker\alpha$ (since $h>0$). 

    If $h\equiv 1$, then $X_1$ is the Reeb vector field of $\alpha$. If $h$ is autonomous, then $X_h$ is the Reeb vector field of $\frac1h\alpha$. Nonautonomous Hamiltonians generate all paths of contactomorphisms. This is a contrast to symplectic dynamics, where there is an obstruction in $H^1(M;\RR)$ for a path of symplectomorphisms to be Hamiltonian, called flux. The set $Cont^+(M,\ker\alpha)$ of positive contactomorphisms consists of all contactomorphisms that are reached through positive paths of contactomorphisms.

\begin{example}
    Continuing the example of the cosphere bundle $S^*Q$ of a compact manifold, we observe that the Hamiltonian flow of the defining starshaped Hamiltonian is a reparametrization of the Reeb flow on the spherization. If the Hamiltonian is $\frac12\|p\|_g^2$ with respect to some Riemannian metric $g$, then the induced flow on the $\frac12$-level $\Sigma_g$ set is the co-geodesic flow on $Q$. Any other fiberwise starshaped hypersurface $\Sigma$ is graphical over $\Sigma_g$ by radial dilation: $\Sigma=f\Sigma_g$, where $f:\Sigma_g\to\RR_{>0}$ and $\alpha_{\Sigma}=f\alpha_{\Sigma_g}$. Thus, studying characteristic flows of fiberwise starshaped hypersurfaces amounts to the same as studying the set of autonomous contact Hamiltonian flows. 
\end{example}

\subsubsection*{Persistent Rabinowitz--Floer homology}
		Given a Liouville domain and asymptotically conical exact Lagrangian $(W,L,\lambda)$ that is bounded by a contact manifold with a Legendrian $(M,\Lambda,\alpha)$, we say a Hamiltonian $h\in\Pcal$ is regular if $\bigcup_{t\neq 0}\varphi_h^t(\Lambda)\pitchfork \Lambda$. We denote the set of regular Hamiltonians by $\Pcal_{\rm reg}$. Note that $\Pcal_{\rm reg}\subseteq\Pcal$ is comeager. For a regular Hamiltonian we can define the action filtered positive Lagrangian Rabinowitz--Floer chain complex $\RFC_+^T(W,L;h)$. The induced homology is the action filtered positive Lagrangian Rabinowitz--Floer homology
			$$\RFH^T_+(W,L;h).$$
		We drop $W,L$ from the notation if there is no possibility of confusion. 		For the analytical details of this homology, especially the discussion of the differential, we refer to~\cite{D18}. Note that the results below on interleavings require that $h\in\Pcal_{\rm reg}$, whereas the results on topological entropy do not. The reason will become clear in the last paragraph, where perturbations of the Legendrian are discussed. 
		
		Positive Lagrangian Rabinowitz--Floer homology has the following properties, cf.~\cite{D18} and the references therein.
		
\begin{enumerate}
	\item The chain complex is a $\ZZ_2$-vector space generated by chords of the contact Hamiltonian vector field $X_h$ that start and end at $\Lambda$ and have length in $(0,T)$. Therefore,
	$$\dim \RFH^{T}_+(h)\leq\#\{\mbox{$X_h$-chords from $\Lambda$ to $\Lambda$ of lenght $\leq T$}\}.$$
	
	\item \label{infinitePmod} The family of vector spaces $\{\RFH^{T}_+(h)\}_{T\in\RR_{>0}\cup\{+\infty\}}$ is a persistence module: it is a direct system directed by morphisms induced by inclusion of generators for $T\leq T'$,
	$$\RFH^T_+(h)\xrightarrow{\iota_{T,T'}}\RFH^{T'}_+(h),$$
	which satisfy $\iota_{T',T''}\circ\iota_{T,T'}=\iota_{T,T''}$, and for each finite $T$ the chain complex $\RFC^T_+(h)$ is finite dimensional. We define the total homology as the direct limit
	$$\RFH_+(h):=\varinjlim\RFH^T_+(h).$$
	We denote by $\iota_\infty$ the limit morphisms $\iota_\infty = \varinjlim\iota_{T,T'}$.
	
	\item \label{continuations}For $h,k\in\Pcal_{\rm reg}$ with $h\leq k$ pointwise, we have continuation morphisms
	$$\RFC^T_+(h)\xrightarrow{\phi^T_{h,k}} \RFC^T_+(k).$$
	Given $h\leq k\leq g$ the morphisms satisfy $\phi^T_{k,g}\circ\phi^T_{h,k}=\phi^T_{h,g}$. These morphisms commute with the morphisms induced by inclusion and induce an isomorphism in the total homology. In other words, the following diagram  commutes.
\begin{center}
  \begin{tikzcd}
    \RFH^T_+(h) \arrow{dd}{\phi^T_{h,k}} \arrow{r}{\iota_{T,T'}} & \RFH^{T'}_+(h) \arrow{dd}{\phi^{T'}_{h,k}} \arrow{dr}{\iota_{\infty}} & \\
		 & & \RFH_+:=\RFH_+(h)\iso\RFH_+(k)\\
		\RFH^T_+(k) \arrow{r}{\iota_{T,T'}} & \RFH^{T'}_+(k) \arrow{ur}{\iota_{\infty}} & 
  \end{tikzcd}
\end{center}

	\item\label{cofinal} There is a cofinal subfamily of $\Ccal\subseteq\Pcal_{\rm reg}$ that is closed under scaling, i.e.\ $c\in\Ccal,\lambda>0\Rightarrow \lambda c\in\Ccal$, such that continuation morphisms with respect to scaling coincide with persistent morphisms, i.e.\ for $\lambda>0$
	$$\RFH_+^T(\lambda c)\iso\RFH_+^{\lambda T}(c).$$
	Cofinal means that $\forall h,k\in \Pcal$ there exists $c\in\Ccal$ such that $c\geq h$ and $c\geq k$. 
\end{enumerate}

\begin{remark}
	In the following we assume that $1\in\Pcal_{\rm reg}$, so that every positive constant  is in $\Pcal_{\rm reg}$ and we can choose $\Ccal=\RR_{>0}$. If $1\notin\Pcal_{\rm reg},$ then we choose any regular autonomous Hamiltonian and scalings thereof, for they also admit the scaling property. Alternatively we can choose $\Ccal$ to be the set of regular autonomous Hamiltonians. In dynamical terms, autonomous Hamiltonians parametrize positive paths of contactomorphisms that are flows. 
\end{remark}

\subsubsection*{Interleaving stability of persistence modules}
	Remember that two persistence modules $(V,\iota^V),(W,\iota^W)$ are $\delta$ interleaved if there are morphisms of shifted persistence modules $F:V^t\to W^{t+\delta},G:W^t\to V^{t+\delta}$ such that concatenating shifts of $F$ and $G$ yields the continuation morphisms $G^{t+\delta}\circ F^t=\iota^V_{t,t+2\delta},F^{t+\delta}\circ G^t=\iota^W_{t,t+2\delta}$.
	
	We can combine the diagram in~(\ref{continuations}) with the property in~(\ref{cofinal}) to obtain for $0<c\leq h\leq C$ and for every $T$
\[
	\begin{tikzcd}
    \RFH_+^{T}(c) \arrow{r}{\phi^{T}_{c,h}} & \RFH_+^{T}(h)  \arrow{r}{\phi^{T}_{h,C}} & \RFH_+^{T}(C) \iso \RFH_+^{\frac CcT}(c).
  \end{tikzcd}
\]
If we logarithmize the persistence parameter $T=e^\tau$, concatenating the above diagram for $T=T$ and $T=\frac CcT$ yields a $\log\frac Cc$ interleaving of the two persistence modules $\RFH_+^{e^\tau}(h)$ and $\RFH_+^{e^\tau}(c)$:
\[
	\begin{tikzcd}
    \RFH_+^{T}(c) \arrow{d} & \RFH_+^{\frac CcT}(c) \arrow{d}\\ \RFH_+^{T}(h)  \arrow{ur} & \RFH_+^{\frac CcT}(h).
  \end{tikzcd}
\]
Thus, our choice of $\Pcal$ implies that every induced persistence module is interleaved with the persistence module induced by a constant, with interleaving distance given by the logarithmized oscillation of the Hamiltonian. 

\subsubsection*{Infinite persistence module}
	From the persistence module $\RFH_+^T(h)$ we can define a new one by taking the image in the total homology $\iota_\infty \RFH_+^T(h)\subseteq\RFH_+$, where the persistence morphisms are given by $\iota_{T,T'}^\infty\circ\iota_\infty=\iota_\infty\circ\iota_{T,T'}$ and where continuation morphisms for different $h,k$ are given by $\phi_{h,k}^{\infty,T}\circ\iota_\infty=\iota_\infty\circ\phi^T_{h,k}$. This amounts to deleting all finite bars from the associated barcodes. Since all morphisms commute with $\iota_\infty$, the properties above also hold for $\iota_\infty \RFH_+^T(h)$. The interleaving diagram from before becomes
\[
	\begin{tikzcd}
    \iota_\infty\RFH_+^{T}(c) \arrow{d} \arrow{r}{\subseteq} & \iota_\infty\RFH_+^{\frac CcT}(c) \arrow{d} \\ \iota_\infty\RFH_+^{T}(h) \arrow{ur} \arrow{r}{\subseteq} & \iota_\infty\RFH_+^{\frac CcT}(h).
  \end{tikzcd}
\]
All persistence morphisms $\iota^\infty_{T,T'}$ are inclusions as subspaces of $\RFH_+$ and thus for any $h\in\Pcal$ the dimension of $\iota_\infty \RFH_+^T(h)$ is monotone in $T$. Since the above interleaving diagram commutes and since the horizontal arrows are inclusions, the vertical and diagonal arrows must be injections. We conclude that 
\[
	\dim\iota_\infty \RFH_+^{T}(h)\in \left[ \dim\iota_\infty \RFH_+^T(c),\dim\iota_\infty \RFH_+^{\frac CcT}(c)\right],
\]
which implies that 
\[
	\Gamma(\dim\iota_\infty \RFH_+^{T}(h))\geq  \Gamma(\dim\iota_\infty \RFH_+^{T}(c))=c\Gamma(\dim\iota_\infty \RFH_+^{T}(1))
\]
as claimed. 
\begin{remark}
	Note that the bound from above is not interesting since there might be homologically invisible chords, and since the volume growth of the Legendrian submanifold might not realize the supremum in the definition of volume growth. In rare cases this is not the case, as for geodesic flows of hyperbolic manifolds. 
\end{remark}
	
	On the way we have reproved the following classical corollary for the logarithmized persistence parameter:
\begin{corollary}
	The persistence modules $\RFH^{e^\tau}(h)$ and $\RFH^{e^\tau}(\min h)$ are $\log\frac Cc$-interleaved. Likewise, the persistence  modules $\iota_\infty\RFH^{e^\tau}(h)$ and $\iota_\infty\RFH^{e^\tau}(\min h)$ are $\log\frac Cc$-interleaved.
\end{corollary}

\subsubsection*{Connection with wrapped homology}
	If $1\in \Pcal_{\rm reg}$, then  the positive Lagrangian Rabinowitz--Floer homology of the constant Hamiltonian 1 (which induces the Reeb flow of $\alpha$) is isomorphic to the positive part of wrapped Floer homology, cf.~\cite{D18}. That is, if $0\leq a$, then 
$${\rm WH}^{(0,a)}(W,L)\iso\RFH_+^{a}(1).$$
This isomorphism is analogous to the isomorphism between Rabinowitz--Floer homology and symplectic homology~\cite{CFO10}. 

This is of much interest to our situation, since in some cases wrapped Floer homology is more computable than Lagrangian Rabinowitz--Floer homology. In particular, wrapped Floer homology admits a Pontrjagin product, which can be used to study the dimensional growth $\Gamma^{\rm symp}(W,L)$ of ${\rm WH}^{(0,a)}(W,L)$. This is used in~\cite{AM17} to find examples of contact manifolds different from cotangent bundles such that every Reeb flow has positive topological entropy. By the above isomorphism, also all positive contactomorphisms on these spaces have positive topological entropy.

\subsubsection*{Stability of chord counting under perturbations of the Legendrian}
As the last step in the construction, we want to make a statement about a generic nearby Legendrian of $\Lambda$. For this we take the following result from~\cite[Proposition 1.7]{D18}. 
\begin{proposition}\label{changeLegendrian}
	Let $\Lambda'$ be a Legendrian that is isotopic through Legendrians to $\Lambda$. Let $\psi$ be a contactomorphism that takes $\Lambda$ to $\Lambda'$ so that $(\psi^{-1})^*\alpha=f\alpha$. Then the exponential growth of the number of $\varphi^t$-chords from $\Lambda$ to $\Lambda'$ of length $\leq T$ is at least $\min f \cdot\min h\cdot\Gamma^{\rm symp}(W,L)$.
\end{proposition}

We sketch the proof.
Suppose that a Legendrian $\Lambda$ is perturbed through an isotopy of Legendrians to a nearby Legendrian $\Lambda'$. The Legendrian isotopy can be extended to a path of contactomorphisms $\{\psi_t\}_{t\in[0,1]}$. The $X_h$ chords from $\Lambda$ to $\Lambda'$ are in 1-1 correspondence with $D(\psi_1)^{-1}X_h$ chords from $(\psi_1)^{-1}\Lambda$ to $\Lambda$. 

We can define a new Hamiltonian $g$ which generates $(\psi_t)^{-1}$ for time $t\in[0,1]$ and thus takes $\Lambda$ to $(\psi_1)^{-1}\Lambda$ in time $1$, and then generates $D(\psi_1)^{-1}X_h$. The transformation formula of the contact Hamiltonian tells us that 
$$g(x,t)|_{t\geq 1}=\alpha_x (D(\psi_1)^{-1}X_h(\psi(x),t-1))=f(x) h(\psi(x),t-1)$$ and therefore 
$g|_{t\geq 1}\geq \min f \cdot\min h$. Unfortunately, the Hamiltonian defined this way is not nessessarily smooth at 1 and is not necessarily positive in $[0,1]$. 

Both problems can be solved by convex combination with a Reeb flow, which strongly flows forward in time $[0,1]$ and then is reverted in time $[1,T_0]$ for $T_0$ large enough such that combination with $g$ affects the lower bounds only a little. The chords from $\Lambda$ to $\Lambda$ of the resulting path of contactomorphisms of length $\geq T_0$ are in 1-1 correspondence with  the chords of length $\geq T_0-1$ from $\Lambda$ to $\Lambda'$, where the correspondence shifts the period by 1, which proves the proposition.

If the deformation of the Legendrian is small, then $\min f$ is close to 1. A convenient consequence of this perturbative stability is that for almost all deformations $\Lambda'$ of $\Lambda$ the Hamiltonian $g$ produced in the proof of Proposition~\ref{changeLegendrian} is regular, $g\in\Pcal_{\rm reg}$. Thus, the conclusion of Theorem~\ref{thm2} holds for all Hamiltonians $h\in\Pcal$.

\section{Appendix: Lagrangian non-autonomous Rabinowitz--Floer homology}\label{appendix}

In this appendix, we briefly describe the construction of Lagrangian Rabinowitz--Floer homology for non-autonomous Hamiltonians. Since there are many details that go beyond the scope of this article, we refer the interested reader for proofs and more details to~\cite{D18} and the references therein.

\subsubsection*{Geometric setup}

    As in Section~\ref{sec:poscont}, let $(W,L,\lambda)$ be a Liouville domain with asymptotically conical  Lagrangian filling the contact manifold  with Legendrian $(M,\Lambda,\alpha)$. The completion $(\tiilde W,\tiilde L,\lambda)$ of $(W,L,\lambda)$ is obtained by gluing in $(M\times[1,\infty),\Lambda\times[1,\infty),d(r\alpha))$ using the Liouville vector field. We denote the radial coordinate by~$r$. 

\subsubsection*{The action functional}
Let $H^t:W\times \RR\to \RR$ be a time-dependent smooth function on $W$ (later $H$ is specifically chosen). We define the action functional $\Acal_{H^t}:\Omega(L)\times \RR\to\RR$, where $\Omega(L)$ denotes the Hilbert manifold of $W^{1,2}$ paths $x:[0,1]\to W$ that start and end at $L$, by 
\begin{eqnarray}\label{functional}
	\Acal_{H^t}(x,\eta)&=&\frac1\kappa \left(\int_0^1x^*\lambda-\eta\int_0^1H^{\eta t}(x(t))\;dt\right),
\end{eqnarray}
where $\kappa$ is some positive constant discussed later. A pair $(x,\eta)$ is a critical point of $\Acal_{H^t}$ if and only if it satisfies the equations
\begin{equation*}\label{crit:lrfh}
\left\{\begin{array}{rcl}
		\dot x(t)&=&\eta X_{H^{\eta t}}(x(t)),\\
		H^\eta(x(1))&=&0,
	\end{array}\right.
\end{equation*}
where $X_{H^t}$ is the Hamiltonian vector field generated by $H^t$.
The first equation implies that $x$ is an orbit of $X_{H^t}$, but with time scaled by $\eta$. The second equation, which one deduces by integration by parts, implies that the orbit ends on $(H^{\eta})^{-1}(0)$. If $H^t=H$ is autonomous, then $H^{-1}(0)$ is a hypersurface for which $\eta$ plays the role of a Lagrange multiplier, and $H(x(t))=0$ for all $t$. For time-dependent $H^t$, however, there is no such hypersurface, and $H^{\eta t}(x(t))$ might be very large or very small for $t<1$.

If $H$ depends only on $r$, $H(r)=0$ only for $r=1$ and $H'(1)=1$, then $\Crit\Acal_H$ is the set of Reeb orbits from $\Lambda$ to $\Lambda$ with period $\eta$ (running backwards if $\eta<0$), and $\Acal_H(x,\eta)=\eta$ at critical points. Note that for $\eta=0$ the critical points are constant orbits that form the critical manifold $\Lambda=L\cap H^{-1}(0)$. Thus, $\Acal$ is never Morse. If $(W,L)$ is regular, all critical points with $\eta\neq0$ are regular and the critical manifold at $\eta=0$ is Morse--Bott. Since we only have one nontrivial critical manifold, we do not focus on the Morse--Bott situation. We choose a Morse function on $\Lambda$ and abusing notation we denote by $\Crit\Acal_H$ the union of the isolated critical points of~$\Acal_H$ with the critical points of the Morse function. The action of a critical point of the Morse function is defined to be the value of the functional $\Acal_H$ on the critical component, i.e.~$0=\Acal_H(\Lambda)$.

\subsubsection*{The choice of Hamiltonian}

For the study of positive contactomorphisms~$\varphi$ it is crucial to carefully construct Hamiltonians~$H^t$ on $W$ in such a way that 
\begin{itemize}
	\item critical points of $\Acal_{H^t}$ encode dynamical information on $\varphi$,
	\item the resulting homology is well defined (i.e.\ there are no problems with compactness and transversality),
	\item the continuation morphisms for monotone deformations of $H^t$ are monotone with respect to the action.
\end{itemize}

Our main objective is to study a positive contactomorphism~$\varphi$ of $(M,\alpha)$. To do so, choose a positive path of contactomorphisms $\varphi^t$ such that $\varphi^0=id,\;\varphi^1=\varphi$. By convex combination with a Reeb flow (see~\cite[Proposition 6.2]{FLS15} for more details), $\varphi^t$ can be deformed with fixed endpoints to $\tilde\varphi^t$ such that $\tilde\varphi^t$ is the Reeb flow of $\alpha$ for $t$ near $0$ and $1$. This means that the contact Hamiltonian $h^t$ on $M$ generating~$\varphi^t$ is constant $\equiv 1$ for $t$ near $0$ and $1$. This deformation can be performed such that the order is preserved: $\tilde h_1^t\leq\tilde h_2^t$ whenever $h_1^t\leq h_2^t$. From now on we assume that this deformation is already performed if not stated differently. Consequently, the concatenation of positive contactomorphisms is a positive contactomorphism. 

Furthermore, the fact that $h^t$ is constant near $0$ and $1$ allows us to extend $h^t$ to $t\in\RR$ periodically or constantly. We always choose a mixed extension: $h^t\equiv1$ for $t\leq0$, and for positive time we choose an extension such that $h^t$ and its derivative are uniformly bounded. In the actual applications the choice will be that $h^t$ is periodic for $t\geq0$ and thus the longterm behavior of $\varphi^t_{h^t}$ coincides with the behavior of iterations of $\varphi$ since $\varphi^k_{h^t}=\varphi^k, k\in\NN$. The choice for negative time is necessary for our proof of monotonicity of continuation morphisms. 
%

We collect our assumptions on the Hamiltonian:
\begin{assumption}\label{classofh}
	The contact Hamiltonian $h^t:M\times\RR\to\RR$ satisfies for some positive constants $c,C,c'$
	\begin{itemize}
		\item $h^t\equiv 1$ for $t\leq0$,
		\item $0<c\leq h^t\leq C$ and $|\frac d{dt}h^t|\leq c'$,
		\item  $\bigcup_{t\neq 0}\varphi^t_{h^t}\Lambda$ and $\Lambda$ intersect transversely.
	\end{itemize}
\end{assumption}

To work in the Liouville domain we construct from the contact Hamiltonian $h^t$ on $M$ a Hamiltonian $H^t$ on $\haat W$ in a uniform way, depending on two large enough parameters $\kappa$ and $K$, which we choose for every finite action window individually. The flow lines of $X_{H^t}$ are lifts of $\varphi^t$-flow lines if we set on $M\times\RR^{> 0}$
\begin{equation}
H^t:=rh^t-\kappa.
\label{primitivesH}
\end{equation}
 For such an $H^t$ the critical points of $\Acal_{H^t}$ end in $\{r=\kappa/h^\eta\}$. Changing $\kappa$ does not change critical points in an essential way (provided they do not run into $r=0$), but translates them in the $r$-direction. In order to have a smooth Hamiltonian on $\haat W$ and to get compactness of moduli spaces later on, we deform $H^t$ to depend only on $r$ but not on $t$ for $r\leq1-\delta$ and $r\geq\kappa K+1$. To make this precise, we choose independently of the action window constants $\mu_-\leq\min\{h^t(x)\mid x\in M\}, \mu_+\geq\max\{h^t(x)\mid x\in M\}$ which we gather in a function $\mu(r)=\mu_-$ for $r\leq 1$ and $\mu(r)=\mu_+$ for $r\geq \kappa K$, a convex smooth function $\rho:\RR^{\geq 0}\to\RR,$ that will play the role of the radius smoothed out over $W$, and a smooth function $\beta:\RR^{\geq0}\to[0,1]$, that will serve as a transition parameter, such that


\begin{eqnarray*}
	\rho(r)&=&\begin{cases} 1-\frac23\delta &\quad\mbox{if}\quad r\leq1-\delta,\\r&\quad\mbox{if}\quad r\geq 1,\end{cases}\\
	\beta(r)&=&\begin{cases}0&\quad\mbox{if}\quad r\leq1-\delta \quad \mbox{or}\quad r\geq \kappa K+1,\\1&\quad\mbox{if}\quad 1\leq r\leq \kappa K,\end{cases}\\
	\beta'(r) &\in&\begin{cases} [0,\frac2\delta]&\quad\mbox{if}\quad r\leq1,\\ [-2,0]&\quad\mbox{if}\quad r\geq \kappa K.\end{cases}
\end{eqnarray*}
Then we define the Hamiltonian
\begin{eqnarray*}
	H^t_{h^t,\kappa,K}(x,r)&=&\rho(r)\Big(\beta(r)\,h^t(x)+(1-\beta(r))\,\mu\Big)-\kappa.
\end{eqnarray*}
The factor $\frac1\kappa$ in Definition~(\ref{functional}) does not influence the critical points, but only their action values. In fact, the following lemma shows that for $\kappa,K$ large enough, the critical points (up to translation in the $r$-direction) and their actions do not depend on the choice of the constants.

\begin{lemma}\label{independence}
	Let $h^t$ satisfy Assumption~\ref{classofh}. Given $a<b$, there are constants $\kappa_0,K_0$ such that for $\kappa\geq\kappa_0$ and $K\geq K_0$ the following holds. If $(x,\eta)$ is a critical point with $a\leq\Acal_{H^t_{h^t,\kappa,K}}(x,\eta)\leq b$, then the radial component of $x$ stays in $[1,K\kappa]$ for $t\in[0,1]$ and $\Acal_{H^t_{h^t,\kappa,K}}(x,\eta)=\eta$. 
\end{lemma}
\begin{proof} 
	A detailed proof and explicit constants $\kappa_0,K_0$ are given in~\cite[Proposition 4.3]{AF12} in the setup of cotangent bundles. It applies verbatim in the setting of Liouville domains.
\end{proof} 

In the following we abbreviate $\Acal_{H^t_{h^t,\kappa,K}}=\Acal$.



\subsubsection*{Moduli spaces of Floer strips}
We choose an asymptotically conical almost complex structure $J$ on $\haat W$. It induces the following inner product on $\Omega(L)\times\RR$
\begin{equation*}
	\langle(\hat x_1,\hat\eta_1),(\hat x_2,\hat\eta_2)\rangle=\int_0^1\omega(\hat x_1,J\hat x_2)\; dt+\hat\eta_1\hat\eta_2.
\end{equation*}
For this inner product the $L^2$-gradient equation of $\Acal$ is the Rabinowitz--Floer equation
\begin{equation*}
	\begin{cases}
		\partial_s x+J(x)[\partial_t x-\eta X_H(x(s,t))]=0,\\
		\partial_s\eta +\int_0^1H(x(s,t))\;dt=0.
	\end{cases}
\end{equation*}
In addition we choose a Riemannian metric $g$ on the only nontrivial critical manifold $\Lambda\times\{\eta=0\}$. For two critical points $(x_1,\eta_1)$ and $(x_2,\eta_2)$ of $\Acal$ we consider the moduli space $\tiilde\Mcal((x_1,\eta_1),(x_2,\eta_2),H,J)$ of gradient flow lines with cascades from $(x_1,\eta_1)$ to $(x_2,\eta_2)$. For more information on flow lines with cascades, see~\cite{CF09}. We denote the subset where the linearization of the Floer equation has Fredholm index $k$ by 
$$\tiilde\Mcal^k((x_1,\eta_1),(x_2,\eta_2),H,J).$$ On this space there is a natural action by $s$-translation of the domain (if there are multiple cascades, then there is one such action per cascade). We take the quotient by this action to obtain the reduced moduli space of gradient flow lines 
$$\Mcal^{k-1}((x_1,\eta_1),(x_2,\eta_2),H,J)$$ 
(if there are multiple cascades, then the dimensional shift is by the number of cascades). 

For regular $(W,L)$ and generic $J$ and $g$ we have transversality and compactness modulo breaking for these moduli-spaces. 

\subsubsection*{Chain complex}
In view of Lemma~\ref{independence}, for a given choice of constants the homology can only behave well in a finite action window. Let the constants $K,\kappa$ be large enough for the action window $(a,b)$ as in Lemma~\ref{independence}. We define the Rabinowitz--Floer chain groups $\RFC^a(h^t,\kappa,K)$ as free $\ZZ_2$-module generated over $\Crit\Acal$,
\begin{equation*}
	\RFC^a(h^t,\kappa,K)=\sum_{(x,\eta)\in\Crit\Acal,\;\Acal(x,\eta)<a}\ZZ_2\cdot(x,\eta),
\end{equation*}
and the Rabinowitz--Floer chain groups for an action window $\RFC^{(a,b)}(h^t,\kappa,K)$ as the quotient
\begin{equation*}
	\RFC^{(a,b)}(h^t,\kappa,K)=\RFC^b(h^t,\kappa,K)/\RFC^a(h^t,\kappa,K).
\end{equation*}
The differential is defined by counting modulo $\ZZ_2$ isolated Rabinowitz--Floer-strips:
\begin{equation*}
	\partial(x,\eta)=\sum_{(x',\eta')\in\Crit\Acal}\#_{\ZZ_2}\Mcal^0((x,\eta),(x',\eta'),H,J)\cdot (x',\eta').
\end{equation*}
By a gluing argument we can identify $\partial^2$ with counting broken cascades in $\partial\Mcal^1$, whose cardinality is zero modulo $2$. Thus, we can define $\RFH^{(a,b)}(h^t,\kappa,K)$ as the filtered homology of this chain complex. These groups are independent of $\kappa\geq\kappa_0,K\geq K_0$, which is why we denote them by $\RFH^{(a,b)}(h^t)$ for brevity. 

\subsubsection*{Extension of action windows and definition of the homology}

To change the action windows, let $a\leq a', b\leq b'$ and choose constants $\kappa,K$ large enough so that we can use the same functional for both action windows $(a,b)$ and $(a',b')$. There are homomorphisms induced by inclusion of generators $\RFH^{(a,b)}(h^t)\to\RFH^{(a',b')}(h^t)$ which are independent of the specific choice of $\kappa,K$. We define $\RFH^{(-\infty,b)}$ as the inverse limit, $\RFH^{(a,\infty)}(h^t)$ as the direct limit and $\RFH(h^t)=\RFH^{(-\infty,\infty)}(h^t)$ as direct inverse limit (in this order, to preserve exactness of long exact sequences), while adjusting $\kappa,K$. The quotient of infinite interval homologies reproduces finite interval homologies: 
$$\RFH^{(a,b)}(h^t)=\RFH^{(-\infty,b)}(h^t)/\RFH^{(-\infty,a)}(h^t).$$
 We denote by 
$$\RFH_+^T(h^t):=\RFH^{(0,T)}(h^t)$$ 
the positive part of the homology and by $\iota_\infty$ the homomorphisms $\iota_\infty:\RFH_+^T(h^t)\to\RFH_+^\infty(h^t)$ induced by inclusion.

\subsubsection*{Invariance properties}
To change from one Hamiltonian to an other one, consider now a family of Hamiltonians $h_s^t, s\in\RR,$ such that $\partial_sh_s^t$ is supported in $s\in[0,1]$. Suppose that for the associated family of functionals $\Acal_s:=\Acal_{h^t_s,\kappa,K}$ the constants $\kappa,K$ are chosen uniformly large enough for $[a,b]$. We set $\Acal_-=\Acal_s$ for $s\leq0$ and $\Acal_+=\Acal_s$ for $s\geq1$, and similarly $h_-^t,h_+^t$. The continuation homomorphism $\Phi:\RFC(h_-^t)\to\RFC(h_+^t)$ is defined by counting the 0-dimensional components of the moduli space of curves $(x_s,\eta_s)$ that satisfy the equation
\begin{eqnarray}\label{conteqn}
	\partial_s(x_s,\eta_s)&=&-\nabla\Acal_s(x_s,\eta_s),
\end{eqnarray}
such that $\lim_{s\to\pm\infty}(x_s,\eta_s)=(x_\pm,\eta_\pm)$ for critical points $(x_\pm,\eta_\pm)$ of $\Acal_\pm$. Then $\Phi$ induces an isomorphism $\RFH(h_-^t)\to\RFH(h_+^t)$, because $\eta$ is bounded along deformations, and actually does not depend on the homotopy $h_s$ but only on the endpoints~$h_\pm$.

Unfortunately, this isomorphism does not respect the action the filtration of the homology. We therefore restrict our attention to monotone deformations, i.e.\ $\partial_sh_s^t(x)\geq0\;\forall\; s,t,x$. The following proposition says that such monotone deformations are compatible with the action filtration. 
In the proof it is essentially used that we extend $h^t$ to be constant for $t\leq0$.

\begin{proposition}[Monotonicity]\label{monotonicity}
	Let $h_-^t\leq h_+^t$ be two time-dependent positive contact Hamiltonians that satisfy Assumption~\ref{classofh}. Then the continuation homomorphism 
	$$\Phi:\RFH(h_-^t)\to \RFH(h_+^t)$$
	restricts for every $a$ to 
	$$\Phi|_{\RFC^{(-\infty,a)}(h_-^t)}:\RFH^{(-\infty,a)}(h_-^t)\to \RFH^{(-\infty,a)}(h_+^t).$$
\end{proposition}

\begin{proof}
See~\cite[Proposition~3.7]{D18}.
\end{proof}

This provides all the tools needed in Section~\ref{sec:poscont}.


\begin{thebibliography}{xxxx}

\bibitem{Abo10}
M.\ Abouzaid.
A geometric criterion for generating the Fukaya category. 
{\it Publ.\ Math.\ Inst.\ Hautes Etudes Sci.\ } {\bf 112} (2010) 191--240.

\bibitem{AAS20}
A.\ Abbondandolo, M.\ Alves, M.\ Sa\u glam, F.\ Schlenk.
Entropy collapse versus entropy rigidity for Reeb and Finsler flows. 
{\it In preparation}. 

\bibitem{Alv16a}
M.\ Alves.
Cylindrical contact homology and topological entropy.
{\it Geom.\ Topol.}~{\bf 20} (2016), 3519--3569.

%

\bibitem{AF12} 
P.\ Albers and U.\ Frauenfelder. 
A variational approach to Givental's nonlinear Maslov index. 
{\it Geom.\ Funct.\ Anal.} {\bf 22} (2012), no.~5. 1033--1050.

\bibitem{AM17}
M.\ Alves and M.\ Meiwes.
Dynamically exotic contact spheres in dimensions $\geq7$.
{\it Commentarii Mathematici Helvetici} {\bf 94 (3)} (2019), 569--622.

\bibitem{AM18}
P.\ Albers and W.\ Merry. 
Orderability, contact non-squeezing, and Rabinowitz Floer homology. 
{\it J. Symp. Geom.} {\bf 16(6)} (2018), 1481--1547.
%

\bibitem{B71}
R.\ Bowen,
Entropy for group endomorphisms and homogeneous spaces,
{\it Trans.\ Amer.\ Soc.\ } {\bf 153} (1971), 401--414.

\bibitem{CF09}
K.\ Cieliebak and U.\ Frauenfelder.
A Floer Homology for exact contact Embeddings.
{\it Pacific J.\ Math.} {\bf 293} (2009) no.~2, 251--316. 

\bibitem{CFO10} 
K.\ Cieliebak, U.\ Frauenfelder and A.\ Oancea. 
Rabinowitz Floer homology and symplectic homology.
{\it Annales scientifiques de l'\'Ecole Normale Sup\'erieure} {\bf 43} (2010), no.~6. 957--1015.


\bibitem{D71}
E.\ I.\ Dinaburg,
Connection between various entropy characterizations of dynamical systems,
{\it Izvestija AN SSSR} {\bf 35} (1971), 324--366.


\bibitem{D18}
L.\ Dahinden.
Positive topological entropy of positive contactomorphisms.
{\it J.\ Symp.\ Geom.} {\bf 18(3)} (2020), 691--732.

\bibitem{D20}
L.\ Dahinden.
Redshift and asymptotic density of the length spectrum of Cauchy hypersurfaces.
{\it In preparation.}

\bibitem{DFPV}
L.\ Diaz, T.\ Fisher, M.\ Pacifico, J.\ Vieitez.
Entropy-expansiveness for partially hyperbolic diffeomorphisms.
{\it Discrete and Continuous Dynamical Systems} {\bf 32(12)} (2012) doi: 10.3934/dcds.2012.32.4195.

\bibitem{EP00}
Y.\ Eliashberg, L.\ Polterovich. 
Partially ordered groups and geometry of contact transformations.
{\it Geom.\ Funct.\ Anal.} {\bf 10(6)} (2000), 1448--1476.

\bibitem{FLS15} 
U.\ Frauenfelder and C.\ Labrousse and F. Schlenk. 
Slow volume growth for Reeb flows on spherizations and contact Bott--Samelson theorems.
 {\it J.\ Topol.\ Anal.}~{\bf 07} (2015), no 3, 407--451.


\bibitem{FraSch06}
U.\ Frauenfelder and F.\ Schlenk.
Fiberwise volume growth via Lagrangian intersections. 
{\it J.~Symplectic Geom.} ~{\bf 4} (2006), 117--148.

\bibitem{GHRS19}
S.\ Gangloff, A.\ Herrera, C.\ Rojas and M.\ Sablik.
On the computability properties of topological entropy: a general approach.
arXiv:1906.01745

\bibitem{Gro78}
M.\ Gromov.
Homotopical effects of dilatation. {\it J.~Differential Geom.}~{\bf 13} (1978) 303--310.

\bibitem{HK}
B.\ Hasselblatt, A.\ Katok.
Introduction to the modern theory of dynamical systems.
{\it Cambridge University Press} (1995).

\bibitem{HKC92}
L.\ Hurd, J.\ Kari, and K.\ Culik. 
The topological entropy of cellular automata is uncomputable.
{\it Ergodic Theory and Dynamical Systems}~{\bf 12} (1992), 255--265.

\bibitem{HSX}
Y.\ Hua, R.\ Radu, Z.\ Xia.
Topological entropy and partially hyperbolic diffeomorphisms.
{\it Ergodic Theory Dynam.\ Systems} {\bf 28 nr.\ 3} (2008), 843--862.

\bibitem{KS96}
S.\ Kolyada, L.\ Snoha,
Topological entropy of nonautonomous dynamical systems.
{\it Random and computational dynamics} {\bf 4(2\&3)} (1996), 205--233.

\bibitem{M79}
A.\ Manning.
Topological Entropy for Geodesic Flows.
{\it Annals of Mathematics} {\bf Vol. 110, Nr. 3} (1979), pp. 567-573

\bibitem{MacSch11}
L.\ Macarini and F.\ Schlenk.
 Positive topological entropy of Reeb flows on spherizations. {\it Math.\ Proc.\ Cambridge Philos.\ Soc.}~{\bf 151} (2011) 103--128.

\bibitem{M02}
J.\ Milnor.
Is entropy effectively computable?
https://www.math.iupui.edu/~mmisiure/open/JM1.pdf

\bibitem{N}
S.\ E.\ Newhouse.
Continuity properties of entropy.
{\it Annals of Mathematics}~{\bf 129} (1989), 215--235.

\bibitem{PRSZ}
L.\ Polterovich, D.\ Rosen, K.\ Samvelyan, J.\ Zhang.
Topological Persistence in Geometry and Analysis.
arXiv:1904.04044

\bibitem{Y} 
Y.\ Yomdin. Volume growth and entropy. 
{\it Israel J.\ Math.}  {\bf 57} (1987) 285--300.




%
%
%

\end{thebibliography}
\end{document}